\documentclass[11pt]{article}
\usepackage[a4paper,margin=28mm]{geometry}
\usepackage{amsmath,amssymb,amsthm,mathtools,booktabs,tikz,needspace}
\usepackage[hidelinks]{hyperref}
\newtheorem{theorem}{Theorem}[section]
\newtheorem{lemma}[theorem]{Lemma}
\newtheorem{proposition}[theorem]{Proposition}
\newtheorem{corollary}[theorem]{Corollary}
\newtheorem*{quotedproposition}{Proposition}
\theoremstyle{remark}\newtheorem{remark}[theorem]{Remark}
\newtheorem{example}{Example}
\newcommand{\R}{\mathbb R}
\newcommand{\MS}{\Delta_{\mathrm{MS}}}

\newcommand{\coeff}{\operatorname{coeff}}
\newcommand{\rank}{\operatorname{rank}}
\newcommand{\coker}{\operatorname{coker}}
\newcommand{\ap}[2]{\langle #1,#2\rangle}
\title{A geometric criterion for the dimension of Morgan--Scott spline spaces $S_{2r}^r(\MS)$}
\author{Xinyu Wu\\{\small\texttt{xinyu97@ustc.edu.cn}}
\and Yi Zheng\\{\small\texttt{zhengyi2023@mail.ustc.edu.cn}}
\and Jiansong Deng\\{\small\texttt{dengjs@ustc.edu.cn}}}

\date{} 
\begin{document}
\maketitle
\begin{abstract}
We establish a geometric dimension criterion for spline spaces on nondegenerate standard Morgan--Scott triangulations with four distinct edge directions at each interior vertex. For every integer $r\ge1$, concurrence of the three auxiliary lines joining corresponding interior and boundary vertices implies that $\dim S_{2r}^r(\MS)$ exceeds the Schumaker lower bound by one. The proof constructs a quadratic form through the six dual spoke points and uses its $r$th power to define a nonzero functional annihilating the image of the smoothing-cofactor map. Diener's upper bound then gives the dimension formula. We also give a matrix representation of this functional and an explicit example in the quadratic case. Combined with the known necessary conditions and the negative-branch exclusion of Deng, Feng, and Kozak, the result yields a necessary and sufficient concurrence criterion for $r>2$.
\end{abstract}
\noindent\textbf{Keywords:} bivariate splines; Morgan--Scott triangulation; smoothing cofactors; dimension.\\
\textbf{Mathematics Subject Classification (2020):} 41A15, 65D07.

\section{Introduction}
Let $\Delta$ be a finite conforming triangulation of a polygonal region $\Omega\subset\R^2$. For integers $d,r\ge0$, define
\[
 S_d^r(\Delta)=\{f\in C^r(\Omega):f|_\tau\in\mathcal P_d
 \text{ for every triangle }\tau\in\Delta\},
\]
where $\mathcal P_d$ is the space of real bivariate polynomials of total degree at most $d$. Derivatives through order $r$ match across shared edges and at vertices, with one-sided limits on the boundary. Determining the dimension of $S_d^r(\Delta)$ is a central problem in multivariate spline theory. Unlike the univariate case, this dimension can depend on the geometric positions of the vertices as well as on the combinatorics of the triangulation.

The seven-triangle Morgan--Scott triangulation $\MS$ is a classical example of this geometry dependence. We study the spaces $S_{2r}^r(\MS)$ and the geometric condition under which their dimension exceeds the Schumaker lower bound \cite{Schumaker1979,Schumaker1984} by one.

For $d=2r$, Diener \cite[Theorem 1, p.~544]{Diener} proved
\[
 B_r\le\dim S_{2r}^r(\MS)\le B_r+1,
\]
where $B_r$ is the Schumaker lower bound, given explicitly in Section~\ref{sec:main}. Write $S=s_1s_2s_3$ and $T=t_1t_2t_3$ in the cyclic barycentric coordinates of the boundary vertices relative to the central triangle. His theorem states that a dimension increase can occur only if $S=T$ for odd $r$, or $S=\pm T$ for even $r$.

Diener's Theorem 2 proves $\dim S_{2r}^r(\MS)=B_r+1$ for every $r\ge1$ when the configuration is symmetric about each of the three auxiliary lines. In \cite[Section 6, Remark (5), p.~550]{Diener}, he established that the parity-dependent conditions are necessary and sufficient for $r=1,2$, and conjectured the same characterization for every $r\ge1$. The conjecture was subsequently verified for $r=3$ \cite{CDXF1998}, whereas the case $r=4$ revealed the failure of its negative-sign branch \cite{DXF1998}. 

Moreover, Deng, Feng, and Kozak \cite[Theorem 1.1]{DFK} proved that, for every even $r>2$, the negative condition $S=-T$ instead gives $\dim S_{2r}^r(\MS)=B_r$. Thus Diener's conjecture fails on the negative branch in these orders. They proposed the revised conjecture that, for $r>2$, the dimension is $B_r+1$ if and only if $S=T$. The case $r=2$, in which either sign gives a dimension increase, remains an exception.

For $r=1$, the quadratic $C^1$ space $S_2^1(\MS)$ already admits geometric descriptions of the concurrence criterion. Whiteley's 1990 draft \cite[Example 4, pp.~18--19]{Whiteley} factors a cofactor determinant and obtains concurrence of the three auxiliary lines as the condition for the existence of a spline that is not a global polynomial. Du, Shi, and Wang \cite[Theorem 2.1]{Du} subsequently characterize $\dim S_2^1(\MS)=7$ by the existence of a conic through the six dual points representing the spoke lines. Using Pascal's theorem and projective duality, they recover the equivalent concurrence criterion in their Theorem 2.3. Nixon, Schulze, and Whiteley \cite[Example 41]{NSW} revisit Whiteley's determinant calculation in the language of cofactor rigidity.

In this paper, we prove that concurrence of the three auxiliary lines implies
\[
 \dim S_{2r}^r(\MS)=B_r+1
\]
for every integer $r\ge1$ on the standard Morgan--Scott triangulation described in Section~\ref{sec:main}. Combined with the known necessary conditions and the negative-branch exclusion of Deng, Feng, and Kozak, this establishes the revised dimension criterion for $r>2$. The proof builds on the dual-conic geometry of the quadratic case. The key observation is that the $r$th power of a quadratic form through the six dual spoke points annihilates the spoke contributions under the differential pairing. Applying the same functional at the three interior vertices cancels the central-edge contributions. We show that the resulting functional is nonzero on the target of the cofactor map, which gives the required lower bound; Diener's upper bound then yields equality.

The paper is organized as follows. Section~\ref{sec:main} describes the MS configuration and states the main result. Section~\ref{sec:cofactor} gives the cofactor formulation and derives the dimension formula. Section~\ref{sec:proof} proves the main theorem by normalizing the coordinates, constructing the dual conic, and defining the annihilating functional. Section~\ref{sec:matrix} gives the matrix representation and a quadratic example. Section~\ref{sec:conclusion} concludes the paper.

\section{Configuration and main result}
\label{sec:main}
Let $A_1,A_2,A_3$ be the interior vertices and $V_1,V_2,V_3$ the boundary vertices. Indices are read modulo $3$. The nine interior edges are the three central edges $A_iA_{i+1}$ and the six spokes $A_iV_j$ with $i\ne j$. The boundary is the triangle $V_1V_2V_3$. We assume these edges form an embedded seven-triangle triangulation with nondegenerate triangles, and that the four edges incident to each $A_i$ have pairwise distinct supporting lines. The auxiliary lines $A_iV_i$ are not edges of the triangulation. The seven triangles are
\[
 [A_1,A_2,A_3],\qquad [V_i,A_{i+1},A_{i+2}],\qquad
 [A_i,V_{i+1},V_{i+2}],\quad i=1,2,3.
\]

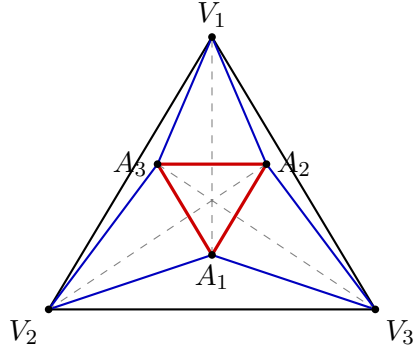
\begin{figure}[ht]
\centering
\begin{tikzpicture}[scale=.72]
\coordinate (V1) at (0,3);\coordinate (V2) at (-3,-2);\coordinate (V3) at (3,-2);
\coordinate (A1) at (0,-1);\coordinate (A2) at (1,.6666667);\coordinate (A3) at (-1,.6666667);
\draw[thick] (V1)--(V2)--(V3)--cycle;
\draw[red!80!black,very thick] (A1)--(A2)--(A3)--cycle;
\draw[blue!75!black,thick] (A1)--(V2) (A1)--(V3) (A2)--(V1) (A2)--(V3) (A3)--(V1) (A3)--(V2);
\draw[dashed,gray] (A1)--(V1) (A2)--(V2) (A3)--(V3);
\foreach \p/\pos/\lab in {V1/above/V_1,V2/below left/V_2,V3/below right/V_3,A1/below/A_1,A2/right/A_2,A3/left/A_3}{\fill (\p) circle (2pt);\node[\pos] at (\p) {$\lab$};}
\end{tikzpicture}
\caption{A concurrent MS configuration. Blue edges are spokes and red edges are central edges. Gray dashed auxiliary lines are not edges of the triangulation; they illustrate the concurrence condition.}
\label{fig:ms}
\end{figure}

Relative to the central triangle, write
\begin{equation}\label{eq:bary}
 V_i=\rho_iA_i+s_iA_{i+1}+t_iA_{i+2},\qquad
 \rho_i+s_i+t_i=1.
\end{equation}
We assume $\rho_i<0$ and $s_i,t_i\ne0$, as in the configuration studied in \cite{Diener,DFK}. No further sign restriction is imposed on $s_i,t_i$. Set
\[
 S=s_1s_2s_3,\quad T=t_1t_2t_3,\quad
 N_r=\binom{2r+2}{2},\quad Q_r=\binom{r+1}{2},
\]
\begin{equation}\label{eq:baseline}
 \sigma_r=\sum_{j\ge1}(r+1-3j)_+,\qquad B_r=N_r+3\sigma_r,
 \quad (t)_+=\max(t,0).
\end{equation}
The sum is finite. Diener's notation $\alpha+\sigma$ corresponds to $N_r+3\sigma_r$ here. The condition $S=T$ is also the concurrence condition for the three auxiliary lines in the projective plane: their equations in the central homogeneous coordinates have coefficient rows $(0,t_1,-s_1)$, $(s_2,0,-t_2)$, and $(t_3,-s_3,0)$, whose determinant is $S-T$. Any two of these rows are independent because all $s_i,t_i$ are nonzero.

The negative branch $S=-T$ has a different incidence interpretation. The intersections $P_i=A_iV_i\cap A_{i+1}A_{i+2}$ have central homogeneous coordinates
\[
 P_1=[0:s_1:t_1],\quad P_2=[t_2:0:s_2],\quad P_3=[s_3:t_3:0].
\]
Their coordinate determinant is
\[
 \det\begin{pmatrix}0&s_1&t_1\\t_2&0&s_2\\s_3&t_3&0\end{pmatrix}=S+T.
\]
Thus $S=-T$ is equivalent to collinearity of these three points. Both incidence statements are understood in the projective plane. The terms positive and negative branch refer to the signs in $S=\pm T$, not to positivity of the products.

\begin{quotedproposition}[Diener \cite{Diener}]
Let $\MS$ be the standard seven-triangle Morgan--Scott triangulation described above, with nondegenerate triangles and no two edges incident to an interior vertex collinear. For every integer $r\ge1$,
\[
 B_r\le\dim S_{2r}^r(\MS)\le B_r+1.
\]
\end{quotedproposition}
These are the bounds of \cite[Theorem 1, p.~544]{Diener}, expressed in our notation. The proof below establishes the lower bound $B_r+1$ under the positive product condition; equality then follows from this upper bound.

\Needspace{9\baselineskip}
The following theorem is the main result of this paper.
\begin{theorem}\label{thm:positive}
Let $\MS$ be a nondegenerate standard seven-triangle Morgan--Scott triangulation with four distinct supporting lines at each interior vertex. With $S,T$ defined by the cyclic barycentric coordinates above, for every integer $r\ge1$,
\[
 S=T\quad\Longrightarrow\quad \dim S_{2r}^r(\MS)=B_r+1.
\]
\end{theorem}

\begin{corollary}\label{cor:criterion}
For the standard Morgan--Scott triangulation in Theorem~\ref{thm:positive}, and for $r>2$,
\[
 \dim S_{2r}^r(\MS)=B_r+1\quad\Longleftrightarrow\quad S=T.
\]
\end{corollary}
\begin{proof}
The forward implication for odd $r$ follows from \cite[Theorem 1]{Diener}. For even $r$, that theorem restricts the increase to $S=T$ or $S=-T$; \cite[Theorem 1.1]{DFK} excludes $S=-T$ for even $r>2$. The reverse implication is Theorem~\ref{thm:positive}.
\end{proof}
\begin{remark}
For $r=2$, both signs can yield a dimension increase; see \cite[Section 6, Remark (5)]{Diener}. Theorem~\ref{thm:positive} concerns only the positive sign in this case.
\end{remark}

\section{Cofactor formulation and the dimension formula}\label{sec:cofactor}
We express the smoothness conditions as cofactor equations and separate the global polynomial part from their solution space. Computing the dimensions of the local images then identifies the dimension increase with the cokernel of the cofactor map.

Write $R_m=\R[x,y,z]_m$ for the homogeneous polynomials of degree $m$. In all polynomial rings below, a negative-degree homogeneous component is understood to be zero. Fix a linear form $h$ defining the original affine plane by $h=1$. Homogenization to degree $2r$ identifies $\mathcal P_{2r}$ with $R_{2r}$; the inverse is restriction to this affine plane. Thus polynomials of degree less than $2r$ are included in this identification. Let $p_+$ and $p_-$ denote the degree-$2r$ homogeneous forms corresponding to the polynomial pieces on the two sides of a shared edge. For any interior edge with supporting line $\ell_e=0$, the $C^r$ continuity condition is
\begin{equation}\label{eq:smooth}
 p_+-p_-=\ell_e^{r+1}q_e,\qquad q_e\in R_{r-1}.
\end{equation}
Indeed, in local affine coordinates with the edge given by $u=0$, the derivatives through order $r$ match if and only if the difference is divisible by $u^{r+1}$. Vanishing on the open edge segment implies the required polynomial coefficient identities on the whole line. Multiplication of $\ell_e$ by a nonzero constant merely rescales $q_e$ and does not affect ranks.

Fix each interior edge form $\ell_e$ once and use the same form in both endpoint equations. Orient the triangular face-adjacency graph. Summing \eqref{eq:smooth} around each interior vertex gives three conformality equations. A common choice of circuit orientation makes every central-edge term appear with opposite signs at its two endpoints. Put
\begin{equation}\label{eq:J}
 J_i=\sum_{e\ni A_i}\ell_e^{r+1}R_{r-1}\subset R_{2r},
 \qquad
 \Phi:R_{r-1}^{\,9}\longrightarrow J_1\oplus J_2\oplus J_3,
\end{equation}
where $\Phi(q)_i=\sum_{e\ni A_i}\varepsilon_{i,e}\ell_e^{r+1}q_e$, with signs determined by the oriented circuits. Thus $\varepsilon_{i,e}=-\varepsilon_{j,e}$ for a central edge joining $A_i$ and $A_j$.

\begin{lemma}\label{lem:global}
There is a linear isomorphism $S_{2r}^r(\MS)\cong R_{2r}\oplus\ker\Phi$.
\end{lemma}
\begin{proof}
Fix an initial triangle $\tau_0$ and use the adjacency-edge orientations already chosen above. A spline determines its polynomial $p_{\tau_0}\in R_{2r}$ on this triangle. For every oriented adjacency edge, the difference between the terminal and initial face polynomials is $\ell_e^{r+1}q_e$ by \eqref{eq:smooth}. Each $q_e$ is unique since multiplication by the nonzero polynomial $\ell_e^{r+1}$ is injective. Summing the differences around each interior vertex gives zero, hence $\Phi(q)=0$. We thus obtain a linear map
\[
 \Psi:S_{2r}^r(\MS)\longrightarrow R_{2r}\oplus\ker\Phi,
 \qquad f\longmapsto(p_{\tau_0},q).
\]

Conversely, take $p_0\in R_{2r}$ and $q\in\ker\Phi$. For a triangle $\tau$, choose a simple adjacency path $\gamma$ from $\tau_0$ to $\tau$ and set
\[
 p_\tau=p_0+\sum_{e\in\gamma}\eta_e\ell_e^{r+1}q_e,
\]
where $\eta_e=1$ or $-1$ according as the path traverses the edge with or against its orientation. This definition is independent of the chosen path. Indeed, the adjacency graph is connected, with seven vertices and nine edges, so its circuit space has dimension $9-7+1=3$. The three circuits around the interior vertices are independent, since each contains a spoke belonging to no other such circuit. They form a circuit basis. Applying this circuit-space statement separately to each monomial coefficient, the condition $\Phi(q)=0$ therefore makes the sum of increments zero on every circuit and hence on the closed walk obtained by following one path and reversing another.

The resulting polynomials satisfy \eqref{eq:smooth} on every shared edge, so their derivatives through order $r$ match there. At a vertex this matching propagates through the adjacent triangles, giving the same derivatives for all incident pieces. They consequently define a $C^r$ spline. The reconstruction is linear and inverse to $\Psi$, which proves the isomorphism. 
\end{proof}

We compute $\dim J_i$ using the standard differential pairing (also called the apolar pairing), the pairing underlying Macaulay inverse systems; see \cite[Section 5]{Schenck}. Here we use ordinary derivatives, without divided-power normalization. Thus the factorials below are the values of the differential pairing on monomials. Here $\R[u,v]_n$ and $\R[U,V]_n$ are the real vector spaces of homogeneous polynomials of degree $n$ in the indicated variables. The uppercase variables are formal variables in a second polynomial ring. We identify this second space with the linear dual of the first by the differential pairing
\[
 \ap{F}{p}=p(\partial_U,\partial_V)F\in\R,
 \qquad F\in\R[U,V]_n,\quad p\in\R[u,v]_n.
\]
The differential operator has the same degree as $F$, so its value is a scalar. On monomials,
\[
 \ap{U^iV^{n-i}}{u^jv^{n-j}}=\delta_{ij}i!(n-i)!.
\]
The pairing is nondegenerate: its matrix in the monomial bases is diagonal with nonzero factorial entries. Thus the map $F\mapsto(p\mapsto\ap{F}{p})$ is an isomorphism
\[
 \R[U,V]_n\longrightarrow (\R[u,v]_n)^*.
\]
In particular, $\ap{F}{p}=0$ means that the linear functional represented by $F$ vanishes at $p$. It does not mean that either polynomial is zero. For a subspace $W\subseteq\R[u,v]_n$, write
\[
 W^\perp=\{F\in\R[U,V]_n:\ap{F}{p}=0\text{ for every }p\in W\}.
\]
Under this dual identification, $\dim W+\dim W^\perp=n+1$.

\begin{lemma}\label{lem:binary}
Let $r\ge1$ and $0\le k\le r-1$ be integers, and let $L_1,\ldots,L_4$ be nonzero, pairwise nonproportional binary linear forms. Put $n=r+1+k$. Then
\[
 \dim\sum_{\nu=1}^4L_\nu^{r+1}\R[u,v]_k
 =n+1-(r-3k-2)_+.
\]
\end{lemma}
\begin{proof}
Let $F\in\R[U,V]_n$. For $L=\alpha u+\beta v$ and $g\in\R[u,v]_k$, multiplication and differentiation are adjoint:
\[
 \ap{F}{L^{r+1}g}=
 \ap{(\alpha\partial_U+\beta\partial_V)^{r+1}F}{g}.
\]
Nondegeneracy in degree $k$ implies that $F$ annihilates $L^{r+1}\R[u,v]_k$ precisely when this directional derivative is zero. Write $(U,V)=t(\alpha,\beta)+\xi w$, where $w$ is independent of $(\alpha,\beta)$. The derivative becomes $\partial_t^{r+1}$, so a homogeneous degree-$n$ polynomial in its kernel has $t$-degree at most $r$ and is divisible by $\xi^{n-r}=\xi^{k+1}$. Equivalently it is divisible by $H^{k+1}$, where $H=\beta U-\alpha V$.

For the four different directions, the corresponding $H_\nu$ are pairwise coprime. Thus the common annihilator is exactly
\[
 (H_1H_2H_3H_4)^{k+1}\R[U,V]_{n-4(k+1)}.
\]
Its dimension is $(n-4(k+1)+1)_+=(r-3k-2)_+$. Subtracting from $\dim\R[u,v]_n=n+1$ proves the formula.
\end{proof}

\begin{proposition}\label{prop:local}
At each interior vertex, $\dim J_i=3Q_r-\sigma_r$.
\end{proposition}
\begin{proof}
Use homogeneous coordinates relative to the central triangle, so that $A_1=[1:0:0]$, $A_2=[0:1:0]$, and $A_3=[0:0:1]$. At $A_1$, all four incident linear forms involve only $y,z$. Decomposing the multiplier by its $x$-degree gives the direct sum
\[
 J_1=\bigoplus_{k=0}^{r-1}x^{r-1-k}
 \left(\sum_{e\ni A_1}\ell_e^{r+1}\R[y,z]_k\right).
\]
Lemma~\ref{lem:binary} and the distinct-direction hypothesis yield
\[
 \dim J_1=\sum_{k=0}^{r-1}\bigl(r+k+2-(r-3k-2)_+\bigr)
 =\frac{3r(r+1)}2-\sigma_r.
\]
The other vertices are treated by permuting coordinates.
\end{proof}

Let $\coker\Phi=(J_1\oplus J_2\oplus J_3)/\operatorname{im}\Phi$. By Lemma~\ref{lem:global} and rank--nullity,
\[
 \dim S_{2r}^r(\MS)=N_r+\dim\ker\Phi
 =N_r+9Q_r-\rank\Phi.
\]
Proposition~\ref{prop:local} gives
\[
 \dim\coker\Phi=\sum_{i=1}^3\dim J_i-\rank\Phi
 =9Q_r-3\sigma_r-\rank\Phi.
\]
Substituting this expression and using $B_r=N_r+3\sigma_r$, we obtain
\begin{equation}\label{eq:cokernel}
 \dim S_{2r}^r(\MS)
 =B_r+\dim\coker\Phi.
\end{equation}
Consequently, to prove the lower bound $B_r+1$, it suffices to construct a nonzero linear functional on $J_1\oplus J_2\oplus J_3$ that vanishes on $\operatorname{im}\Phi$.

\section{Proof of the main theorem}\label{sec:proof}
We now construct the nonzero cokernel functional required by \eqref{eq:cokernel}. Subsection~\ref{sec:normalization} uses the concurrence condition to normalize the coordinates and constructs a quadratic form through the six dual spoke points. Subsection~\ref{sec:functional} takes its $r$th power, proves the annihilation and nonvanishing properties, and completes the dimension argument using Diener's upper bound.

\subsection{Coordinate normalization and the dual conic}\label{sec:normalization}
Under the concurrence condition, a diagonal coordinate change puts the boundary vertices into a convenient form. We then construct a quadratic form vanishing at the six dual spoke points; its power will define the functional in the next subsection.

\begin{lemma}\label{lem:normalize}
If $S=T$, an invertible homogeneous coordinate change gives
\[
 A_1=[1:0:0],\quad A_2=[0:1:0],\quad A_3=[0:0:1],
\]
\[
 V_1=[a:1:1],\quad V_2=[1:b:1],\quad V_3=[1:1:c],
 \qquad abc\ne0.
\]
This change preserves the cofactor equations and spline dimension.
\end{lemma}
\begin{proof}
Initially the three boundary vertices have representatives $(\rho_1,s_1,t_1)$, $(t_2,\rho_2,s_2)$, and $(s_3,t_3,\rho_3)$. Apply $D=\operatorname{diag}(\lambda_1,\lambda_2,\lambda_3)$, where
\[
 \lambda_1=1,\qquad \lambda_3=t_2/s_2,\qquad
 \lambda_2=t_2t_1/(s_2s_1).
\]
Then $\lambda_2s_1=\lambda_3t_1$ and $\lambda_3s_2=\lambda_1t_2$. The condition $S=T$ gives $\lambda_1s_3=\lambda_2t_3$. The transformed representatives are $(\rho_1,\lambda_2s_1,\lambda_3t_1)$, $(t_2,\lambda_2\rho_2,\lambda_3s_2)$, and $(s_3,\lambda_2t_3,\lambda_3\rho_3)$. Divide them by their common second and third, first and third, and first and second coordinates, respectively. The resulting parameters are
\begin{equation}\label{eq:abc}
 a=\frac{\rho_1s_2}{t_2t_1},\qquad
 b=\frac{\rho_2t_1}{s_2s_1},\qquad
 c=\frac{\rho_3t_2}{s_2s_3}.
\end{equation}
All denominators and all $\rho_i$ are nonzero, so $abc\ne0$.

If the original affine chart is $h(v)=1$ and $u=Dv$, use the transported chart $h(D^{-1}u)=1$. The change on homogeneous face polynomials is $p(v)\mapsto p(D^{-1}u)$. It is a linear isomorphism and preserves the divisibility relations \eqref{eq:smooth}. The affine chart is transported together with the coordinates, so the argument is valid regardless of the signs of the diagonal entries.
\end{proof}

For the remainder of the construction, assume $S=T$ and use the coordinates of Lemma~\ref{lem:normalize}. Choose the following linear forms for the six spokes:
\begin{center}
\begin{tabular}{c c c c c c}
\toprule
$A_1V_2$&$A_1V_3$&$A_2V_1$&$A_2V_3$&$A_3V_1$&$A_3V_2$\\
$y-bz$&$cy-z$&$x-az$&$cx-z$&$x-ay$&$bx-y$\\\bottomrule
\end{tabular}
\end{center}
The central edges $A_1A_2,A_2A_3,A_3A_1$ have equations $z,x,y$, respectively. Let $X,Y,Z$ be formal variables in a second ternary polynomial ring. A form in $\R[X,Y,Z]_m$ represents a linear functional on $R_m=\R[x,y,z]_m$ through the differential pairing defined in Subsection~\ref{sec:functional}. In this second ring, define the quadratic form
\begin{equation}\label{eq:C}
 C(X,Y,Z)=aX^2+bY^2+cZ^2
 +(ab+1)XY+(ac+1)XZ+(bc+1)YZ.
\end{equation}
The identities
\begin{align*}
 C(0,Y,Z)&=(bY+Z)(Y+cZ),\\
 C(X,0,Z)&=(aX+Z)(X+cZ),\\
 C(X,Y,0)&=(aX+Y)(X+bY)
\end{align*}
show that $C$ vanishes at each of the six spoke coefficient vectors. For example, $y-bz$ has coefficient vector $(0,1,-b)$, and the first factorization gives $C(0,1,-b)=0$. A nonzero form $\alpha x+\beta y+\gamma z$ determines the dual projective point $[\alpha:\beta:\gamma]$; this is the meaning of a dual spoke point. On the three central coefficient vectors its values are
\begin{equation}\label{eq:nonzeroC}
 C(1,0,0)=a,\qquad C(0,1,0)=b,\qquad C(0,0,1)=c.
\end{equation}
No irreducibility or nonsingularity of $C$ is used in the proof.

\subsection{The annihilating functional and proof of the theorem}\label{sec:functional}
We use the $r$th power of the quadratic form to construct a functional that annihilates the cofactor map. Its nonvanishing on the target gives the lower bound, which is matched by Diener's upper bound.

The space $\R[X,Y,Z]_{2r}$ consists of homogeneous forms of degree $2r$ in $X,Y,Z$. As in the binary case, it is identified with $R_{2r}^*$ by
\[
 \ap{F}{p}=p(\partial_X,\partial_Y,\partial_Z)F\in\R.
\]
On monomials, the pairing is
\[
 \ap{X^iY^jZ^k}{x^{i'}y^{j'}z^{k'}}
 =\delta_{ii'}\delta_{jj'}\delta_{kk'}i!j!k!.
\]
Thus $C$ is a quadratic form in the second polynomial ring, while $F=C^r\in\R[X,Y,Z]_{2r}$ represents the linear functional $\Lambda_r(p)=\ap{F}{p}$ on $R_{2r}$. The condition $\Lambda_r(p)=0$ means that this functional vanishes at $p$; its vanishing on an entire subspace is the annihilation property used below.

\begin{lemma}\label{lem:annihilate}
Let $F=C^r$. If $\ell=\alpha x+\beta y+\gamma z$ and $C(\alpha,\beta,\gamma)=0$, then
\[
 \Lambda_r(\ell^{r+1}g)=0\qquad(g\in R_{r-1}).
\]
\end{lemma}
\begin{proof}
Put $v=(\alpha,\beta,\gamma)$, $\mathbf X=(X,Y,Z)$, and $D_v=\alpha\partial_X+\beta\partial_Y+\gamma\partial_Z$. Since $C$ is quadratic and $C(v)=0$,
\[
 C(\mathbf X+tv)=C(\mathbf X)+tL_v(\mathbf X)
\]
for a linear form $L_v$. Therefore $F(\mathbf X+tv)=(C(\mathbf X)+tL_v(\mathbf X))^r$ has degree at most $r$ in $t$. By the directional derivative identity,
\[
 D_v^{r+1}F(\mathbf X)=
 \left.\frac{d^{r+1}}{dt^{r+1}}F(\mathbf X+tv)\right|_{t=0}=0.
\]
Multiplication and differentiation are adjoint under the differential pairing, as in Lemma~\ref{lem:binary}. Hence, for every $g\in R_{r-1}$,
\[
 \Lambda_r(\ell^{r+1}g)=\ap{F}{\ell^{r+1}g}
 =\ap{D_v^{r+1}F}{g}=0.
\]
\end{proof}

\begin{proof}[Proof of Theorem~\ref{thm:positive}]
Assume $S=T$ and use the normalized coordinates above. Define a linear functional on $J_1\oplus J_2\oplus J_3$ by
\begin{equation}\label{eq:Z}
 Z_r(f_1,f_2,f_3)=\Lambda_r(f_1)+\Lambda_r(f_2)+\Lambda_r(f_3),
 \qquad f_i\in J_i.
\end{equation}
For an edge $e$, let $\iota_e g\in R_{r-1}^{\,9}$ have component $g\in R_{r-1}$ in the factor indexed by $e$ and zero elsewhere, and put $p_e=\ell_e^{r+1}g$ for every interior edge. Figure~\ref{fig:ms} distinguishes the two types of contributions. For a spoke incident to $A_i$, the only nonzero target component is $\varepsilon_{i,e}p_e$, and Lemma~\ref{lem:annihilate} gives
\begin{equation}\label{eq:spokezero}
 Z_r(\Phi(\iota_e g))=\varepsilon_{i,e}\Lambda_r(p_e)=0
 \qquad(e=A_iV_j,\ i\ne j).
\end{equation}
For a central edge $e=A_iA_j$, the endpoint signs satisfy $\varepsilon_{j,e}=-\varepsilon_{i,e}$. Consequently,
\begin{equation}\label{eq:centralcancel}
 Z_r(\Phi(\iota_e g))
 =\varepsilon_{i,e}\Lambda_r(p_e)-\varepsilon_{i,e}\Lambda_r(p_e)=0.
\end{equation}
Every cofactor vector is a sum of such single-edge vectors, so \eqref{eq:spokezero} and \eqref{eq:centralcancel} imply
\begin{equation}\label{eq:ZPhi}
 Z_r\circ\Phi=0.
\end{equation}

To show that the restriction of $Z_r$ to the target space is nonzero, use the central edge $A_2A_3$, whose equation is $x=0$. We have
\begin{equation}\label{eq:witness}
 x^{2r}=x^{r+1}x^{r-1}\in J_2,\qquad
 Z_r(0,x^{2r},0)=(2r)!a^r\ne0.
\end{equation}
Here $(0,x^{2r},0)$ is an element of the target space used to test the functional; it is not the contribution of a single central edge. For that edge, $\Phi$ produces two opposite components, whose evaluations cancel as in \eqref{eq:centralcancel}.

By \eqref{eq:ZPhi} and \eqref{eq:witness}, the nonzero functional $Z_r$ annihilating $\operatorname{im}\Phi$ descends to a nonzero functional $\overline Z_r$ on $\coker\Phi$. Hence \eqref{eq:cokernel} implies
\[
 \dim S_{2r}^r(\MS)\ge B_r+1.
\]
Diener's bound \cite[Theorem 1, p.~544]{Diener}, applied to the original embedded configuration, is $\dim S_{2r}^r(\MS)\le B_r+1$. Its noncollinearity and geometric hypotheses are included in Section~\ref{sec:main}. Together the inequalities prove the theorem. In particular $\dim\coker\Phi=1$, and $\overline Z_r$ spans its dual.
\end{proof}

\section{Matrix representation and an example}\label{sec:matrix}
Throughout this section, $S=T$ and the normalized coordinates of Lemma~\ref{lem:normalize} are used. This section expresses the constructed functional as a nonzero left-kernel vector of the cofactor matrix and explains how to form that matrix. This gives a coordinate realization of the preceding proof.

For the matrix representation, fix a monomial order in $R_{2r}$ and use a monomial basis in each of the nine factors $R_{r-1}$ in the domain. Write
\[
 F=\sum_{i+j+k=2r}c_{ijk}X^iY^jZ^k,
 \qquad (w_r)_{ijk}=i!j!k!c_{ijk}.
\]
Here $\coeff(p)$ is the column of coefficients of $p$ in the chosen monomial order. Thus $\Lambda_r(p)=w_r\coeff(p)$. For each $J_i$, let $T_i$ have as columns the monomial coefficient vectors of a basis of $J_i$. If $M$ is the matrix of $\Phi$ in these bases and $M_{\rm mono}$ uses full monomial coordinates in each component, then
\[
 P=\operatorname{diag}(T_1,T_2,T_3),\quad M_{\rm mono}=PM.
\]
The row representing $Z_r$ in the full monomial coordinates is $z_{\rm mono}=(w_r\mid w_r\mid w_r)$. Since $Z_r\circ\Phi=0$, its matrix expression is
\[
 z_{\rm mono}M_{\rm mono}=0.
\]
Passing to the chosen bases of the actual target spaces gives
\[
 z=z_{\rm mono}P=(w_rT_1\mid w_rT_2\mid w_rT_3),
 \qquad zM=z_{\rm mono}PM=z_{\rm mono}M_{\rm mono}=0.
\]
Moreover $z\ne0$ by \eqref{eq:witness}. It is this nonvanishing after restriction to $J_1\oplus J_2\oplus J_3$ that certifies a nonzero cokernel functional.
Put $h_i=\dim J_i=3Q_r-\sigma_r$ and $H_r=h_1+h_2+h_3$. Then
\[
 T_i\in\R^{N_r\times h_i},\qquad
 M_{\rm mono}\in\R^{3N_r\times9Q_r},\qquad
 M\in\R^{H_r\times9Q_r}.
\]
To construct the matrices, for each edge $e$ and each degree-$(r-1)$ monomial $m$, first expand $p=\ell_e^{r+1}m$ and form its full monomial coefficient column $\mathbf b=\coeff(p)$. In $M_{\rm mono}$, a spoke column has one nonzero vertex block, while a central-edge column has blocks $\mathbf b$ and $-\mathbf b$ at its two endpoints, up to the common sign determined by the circuit orientations.

Next convert each nonzero block to coordinates in the chosen basis of $J_i$: let $\mu_i$ be the unique column satisfying
\[
 T_i\mu_i=\mathbf b.
\]
A solution exists because $p\in J_i$, and is unique because the columns of $T_i$ form a basis of $J_i$. The block inserted into $M$ is $\mu_i$, with the appropriate sign, not $\mathbf b$. In particular, a central edge contributes $\mu_i$ and $-\mu_j$ in its two endpoint blocks; these are coordinates in two different local bases.

This gives $M_{\rm mono}=PM$. Since $P$ is injective, the matrices have the same rank, but their left-kernel dimensions differ by $3N_r-H_r$. More precisely, $\dim\ker M^{\mathsf T}=H_r-\rank M=\dim\coker\Phi$, which is the dimension increase in \eqref{eq:cokernel}. The nonvanishing of $Z_r$ on $J_1\oplus J_2\oplus J_3$ is what ensures $z\ne0$.

\begin{example}[The quadratic case]
For $r=1$, choose the local bases
\[
 J_1:(y^2,yz,z^2),\quad J_2:(x^2,xz,z^2),\quad
 J_3:(x^2,xy,y^2).
\]
The left-kernel vector is
\[
 z=(2b,bc+1,2c\mid 2a,ac+1,2c\mid 2a,ab+1,2b).
\]
For example, its first block on the spoke $y-bz$ gives
\[
 (2b,bc+1,2c)(1,-2b,b^2)^T
 =2b-2b(bc+1)+2cb^2=0.
\]
The central edge $z=0$ contributes equal and opposite endpoint evaluations $2c$. The higher-order construction follows the same pattern, with the single replacement $C\mapsto C^r$ and the corresponding factorial weights.

\end{example}

\section{Conclusion}\label{sec:conclusion}
The concurrence criterion for the quadratic $C^1$ Morgan--Scott space extends to $S_{2r}^r(\MS)$ for every $r\ge1$: concurrence of the three auxiliary lines implies a dimension increase of exactly one above the Schumaker lower bound. This establishes the positive branch of Diener's conjecture beyond completely symmetric configurations. Together with the known necessary conditions and the negative-branch exclusion of Deng, Feng, and Kozak, it gives their revised necessary and sufficient criterion for $r>2$; $r=2$ retains the additional negative branch.

The construction explains how the quadratic geometry persists in higher orders. A conic through the six dual spoke points gives a form $C$ whose power $C^r$ annihilates each spoke contribution under the differential pairing. The central-edge contributions cancel between their two endpoints. Nonvanishing on the target space yields the required cokernel functional, and Diener's upper bound fixes the dimension increase at one. The matrix representation makes this functional explicit as a nonzero left-kernel vector.

The result concerns the same nondegenerate standard MS configuration, with four distinct edge directions at each interior vertex, as the cited dimension bounds. Degenerate triangles and repeated directions require separate local dimension calculations. The present construction gives a dual certificate for the dimension increase.

\section*{Acknowledgements}
\begingroup\small\emergencystretch=1em
This work was supported by the National Key R\&D Program of China (Nos. 2022YFA1005200 and 2022YFA1005201), the NSF of China (No. 12171453), and the Major Project of Science and Technology Innovation Tackling Plan of Anhui Province (No. 202423e09050003).
\par\endgroup

\section*{Declarations}
The authors declare that they have no conflict of interest.

\end{document}